\documentclass[12pt]{amsart}

\usepackage{amssymb}
\usepackage{bm}
\usepackage{graphicx}
\usepackage{psfrag}
\usepackage{color}
\usepackage{url}
\usepackage{algpseudocode}
\usepackage{fancyhdr}
\usepackage{xy}
\input xy
\xyoption{all}

\newtheorem{theorem}{Theorem}[section]
\newtheorem{proposition}[theorem]{Proposition}
\newtheorem{lemma}[theorem]{Lemma}

\theoremstyle{definition}
\newtheorem{definition}[theorem]{Definition}
\newtheorem{example}[theorem]{Example}
\numberwithin{equation}{section}

\newcommand\ff{\mathsf{F}}
\newcommand\sft{\mathsf{T}}
\newcommand\nn{\mathbb{N}}

\newcommand\rr{\mathbb{R}}
\newcommand\zz{\mathbb{Z}}

\begin{document}

	\mbox{}
	\title{The Chain Group of a Forest}
	\author{Felix Gotti}
	\address{Mathematics Department\\UC Berkeley\\Berkeley, CA 94720}
	\email{felixgotti@berkeley.edu}
	\author{Marly Gotti}
	\address{Mathematics Department\\University of Florida\\Gainesville, FL 32611}
	\email{marlycormar@ufl.edu}
	\date{\today}
	
	\begin{abstract}
		For every labeled forest $\ff$ with set of vertices $[n]$ we can consider the subgroup $G$ of the symmetric group $S_n$ that is generated by all the cycles determined by all maximal paths of $\ff$. We say that $G$ is the chain group of the forest $\ff$. In this paper we study the relation between a forest and its chain group. In particular, we find the chain groups of the members of several families of forests. Finally, we prove that no copy of the dihedral group of cardinality $2n$ inside $S_n$ can be achieved as the chain group of any forest. 
	\end{abstract}
	
	\maketitle
	
	\section{Introduction} \label{sec:intro}
	
	It is typical in Mathematics to use intrinsic information of discrete objects such as graphs, trees, and finite posets, to carry out algebraic and geometric constructions. For instance, such constructions include the fundamental group of a graph \cite[Chapter~11]{aH01}, the incidence algebra of a finite poset \cite[Chapter~3]{rS11}, and the forest polytope of a graph \cite[Chapter~50]{aS03}. In this paper we use the maximal paths of a forest $\ff$ to construct a finite group, which we call the \emph{chain group} of $\ff$. The method we use to produce the chain group of a given forest is motivated in part by the way Stanley defines a chain polytope from a locally finite poset (see \cite{rS86}).
	
	Given a finite poset $P = \{x_1, \dots, x_n\}$, its corresponding \emph{chain polytope} $\mathcal{C}(P)$ is defined to be the set of points $(y_1, \dots, y_n) \in \rr_{\ge 0}^n$ satisfying the condition
	\begin{equation} \label{eq:chain polytope condition}
		y_{i_1} + \dots + y_{i_k} \le 1 \ \text{ whenever } \ x_{i_1} <_P \dots <_P x_{i_k} \ \text{ is a maximal chain of } P.
	\end{equation}
	In other words, the chain polytope $\mathcal{C}(P)$ is the intersection of the half-spaces determined by the maximal chains of $P$ as indicated in \eqref{eq:chain polytope condition}.
	
	Let us see how to reuse the same method Stanley applies to build the chain polytope of a poset, to naturally associate a finite group $\mathcal{G}(\ff)$ to each forest $\ff$. Instead of taking $\rr^n$ as the universe containing the half-spaces utilized in \eqref{eq:chain polytope condition} to produce $\mathcal{C}(P)$, we can rather consider $S_n$ as the universe containing the generators of a group $\mathcal{G}(P)$, which is defined by
	\[
		\mathcal{G}(\ff) = \big\langle (s_{i_1} \ \dots \ s_{i_k}) \in S_n \ \text{ whenever } \ x_{i_1} <_P \dots <_P x_{i_k} \ \text{ is a maximal path in} \ \ff \big\rangle.
	\]
	We call $\mathcal{G}(\ff)$ the \emph{chain group} of the forest $\ff$.
	
	Chain polytopes, as introduced in \cite{rS86} by Stanley, have nice features. For example, if $\mathcal{C}(P)$ is the chain polytope of the poset $P$, then the number of vertices of $\mathcal{C}(P)$ equals the number of antichains of $P$; see \cite[Theorem~2.2]{rS86}. In addition, the volume of $\mathcal{C}(P)$ is determined by the combinatorial structure of $P$; see \cite[Corollary~4.2]{rS86}. We will see in Section~3 that the chain group of a forest has a nice behavior; for example, disjoint unions of forests become direct sums of groups (see Proposition~\ref{prop:chain group of disjoint graphs}). On the other hand, if we relabel a given forest $\ff$, the resulting forest has chain group conjugate to $\mathcal{C}(\ff)$ (see Proposition~\ref{prop:order groups of isomorphic and dual posets}).
	
	
	
	There are a few natural questions we might ask about this assignment. How the chain groups of two distinct labeling of the same forest are associated? If $G$ is the chain group of the forest $\ff$, can we determine whether $G$ satisfies certain properties only by studying $\ff$? It is our intension to answer such questions here.
	
	In addition, we might wonder, for a fixed $n$, which subgroups of $S_n$ will show as a chain group of some forest labeled by $[n]$. Given that, every finite subgroup is a subgroup of $S_n$ for $n$ large enough, this is not a question that we expect to answer in its full generality. However, we might hope to decide whether relatively simple families of subgroups of $S_n$ can be realized as chain groups of some $n$-forest. For example, is the alternating group $A_n$ a chain group of an $n$-forest for every $n \in \nn$? This question, along with other similar ones, will be answered later in the sequel.
	
	This paper is structured as follows. In Section~\ref{sec:Background and Notation} we review the definitions on graph theory we will be using later. Then, in Section~\ref{sec:general observations}, we prove that passing from a forest to its chain group behaves well with respect to relabeling and changes disjoint union for direct sum. In Section~\ref{sec:abelian case} we study the abelian chain groups. In Section~\ref{sec:the chain group of a tree} we compute the chain groups of members of several families of trees. We also provide some results useful to find the chain groups of some forests. Finally, in Section~\ref{sec:missing chain groups}, we show that the dihedral cannot be achieved as the chain group of any forest.
	
	\section{Background and Notation} \label{sec:Background and Notation}
	
	In this section, we fix notation and briefly recall the definitions of the main objects related to those being studied here. We also state the relevant properties of such objects necessary to follow the present paper. For background material in group theory, symmetric groups, and graph theory we refer the reader to Rotman \cite{jR94}, Sagan \cite{bS01}, and Bondy and Murty \cite{BM08}, respectively.
	
	The double-struck symbols $\mathbb{N}$ and $\mathbb{N}_0$ denote the sets of positive integers and non-negative integers, respectively. For $n \in \nn$, we denote the set $\{1,\dots, n\}$ just by $[n]$. Following the standard notation of group theory, we let $S_n$ and $A_n$ denote the symmetric and the alternating group on $n$ letters, respectively. In addition, the dihedral group of order $2n$ is denoted by $D_{2n}$.
	
	To settle down our nomenclature, let us recall some basic definitions concerning graphs. A \emph{graph} is a pair $\mathsf{G} = (V,E)$, where $V$ is a finite set and $E$ is a collection of $2$-element subsets of $V$. The elements of $V$ are called \emph{vertices} of $\mathsf{G}$ while the elements of $E$ are called \emph{edges} of $\mathsf{G}$. It is often convenient to denote the set of vertices and the set of edges of $\mathsf{G}$ by $V(\mathsf{G})$ and $E(\mathsf{G})$, respectively. The \emph{degree} of a vertex $v$, denoted by $\deg(v)$, is the number of edges containing it. We say that a vertex is a \emph{leaf} if it has degree one. An edge $\{v,w\}$ is also denoted by $vw$. Distinct vertices $v$ and $w$ of $V$ are called \emph{adjacent} if $vw \in E$. In the context of this paper, a \emph{walk} $\omega$ in $\mathsf{G}$ is a sequence of vertices, say $v_0, \dots, v_\ell$ such that $v_{i-1}$ is adjacent to $v_i$ for each $i = 1, \dots, \ell$. If $v_\ell = v_0$, then the walk $\omega$ is said to be \emph{closed}. If, in addition, $v_i = v_j$ implies that $i = j$ or $\{i,j\} = \{0,\ell\}$ then $\omega$ is called a \emph{path}; in this case we say that the \emph{length} of $\omega$ is $\ell$. A path of $\mathsf{G}$ is \emph{maximal} if it is not strictly contained in another path. A closed path of length at least three is called a \emph{cycle}. A graph is said to be \emph{connected} if any two distinct vertices can be connected by a path. Every graph $G$ is the finite disjoint union of connected graphs, which are called \emph{connected components} of $G$. On the other hand, a graph is called \emph{acyclic} provided it does not contain any cycle.
	
	\begin{definition}
		An acyclic connected graph is called a \emph{tree}. A finite disjoint union of trees is said to be a \emph{forest}.
	\end{definition}

	\begin{example}
		The next figure illustrates a graph $\mathsf{G}$ having four connected components. The leftmost component is a \emph{chain}, the second component is a cycle, the third component is a star, and the fourth component is a tree. Notice that $\mathsf{G}$ is not a forest.
		\begin{figure}[h]
			\centering
			\includegraphics[width = 10.0cm]{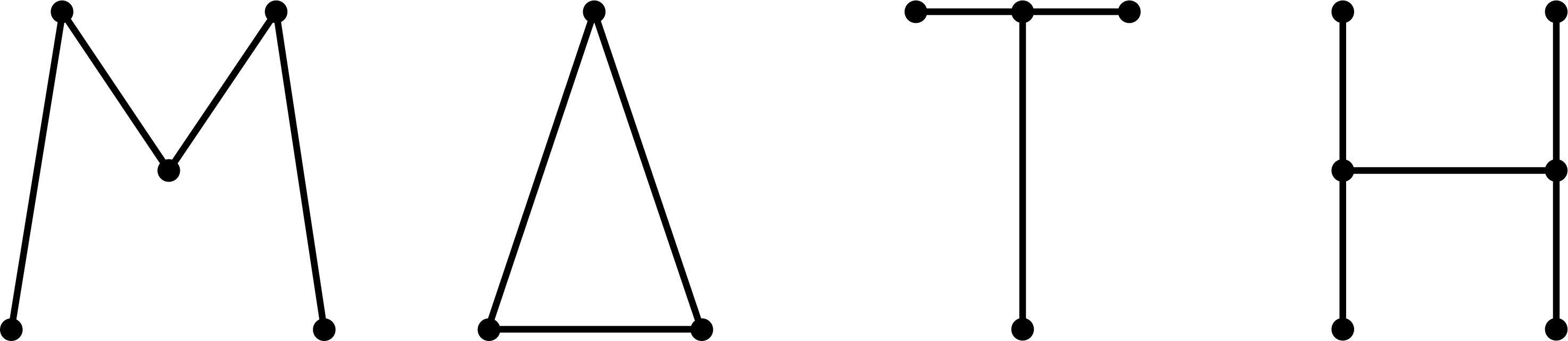}
			\caption{A graph with four connected components.}
			\label{fig:a graph with four connected components}
		\end{figure}
	\end{example}

	A labeled \emph{forest} is a forest $\mathsf{F}$, whose vertices are labeled by the set $\{1,\dots, |V(\mathsf{F})|\}$. All the forests we will be interested in throughout this paper are labeled.
	
	\section{General Observations} \label{sec:general observations}
	
	In this section we formally define the chain group of a forest and explore some general facts connecting them. We also present some examples to illustrate the connection.
	
	\begin{definition}
		For $n \in \nn$ let $\ff$ be a labeled forest with $n$ vertices. The \emph{chain group} of $\ff$, which we denote by $G_{\ff}$, is the subgroup of $S_n$ generated by all cycles $(i_1 \ \dots \ i_m)$ such that $i_1, \dots, i_m$ is a maximal path in $\ff$.
	\end{definition}


	\begin{example}
		Figure~\ref{fig:chain groups of three forests} shows three forests $\ff_1$, $\ff_2$, and $\ff_3$.
		\begin{figure}[h]
			\centering
			\includegraphics[width = 12.0cm]{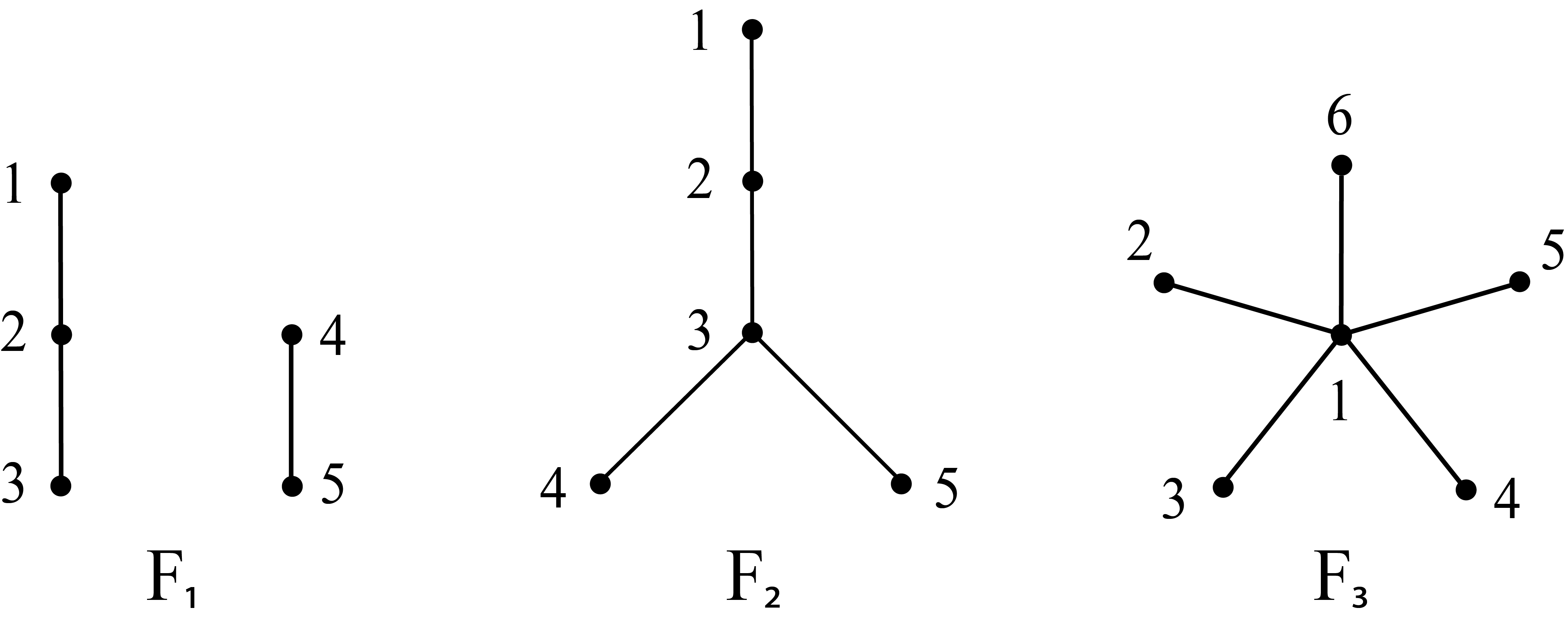}
			\caption{Three labeled forests with their respective chain groups.}
			\label{fig:chain groups of three forests}
		\end{figure}
		The forest $\ff_1$ consists of only two disjoint maximal paths, namely $(1,2,3)$ and $(4,5)$; therefore $G_{\ff_1} = \langle (1 \ 2 \ 3), (4 \ 5) \rangle \cong \zz_3 \times \zz_2$ (cf. Proposition~\ref{prop:chain group of disjoint graphs} below).
		
		On the other hand, $\mathsf{F}_2$ has exactly three maximal paths, which are $(1,2,3,4)$, $(1,2,3,5)$, and $(4,3,5)$. As $(1 \ 2 \ 3 \ 4) \circ (1 \ 2 \ 3 \ 5) = (1 \ 3 \ 5 \ 2 \ 4)$, the chain group $G_{\mathsf{F}_3}$ contains a $5$-cycle. On the other hand, as $(3 \ 4 \ 5) \circ (1 \ 2 \ 3 \ 4) = (1 \ 2 \ 4) (3 \ 5)$ it follows that $G_{\mathsf{F}_2}$ also contains the two cycle $(3 \ 5)$. Hence $S_5 = \langle (1 \ 3 \ 5 \ 2 \ 4), (3 \ 5) \rangle \le G_{\mathsf{F}_2}$, and so $G_{\mathsf{F}_2} = S_5$.
		
		Finally, $\ff_3$ has $\binom{5}{2}$ maximal paths. The chain group of $\ff_3$ is generated by the $3$-cycles $(1 \ a \ b)$ for all $a,b \in \{2,3,4,5,6\}$ with $a \neq b$. These $3$-cycles are enough to generate the whole alternating group (see the proof of Theorem~\ref{thm:chain group of the star graph} for more details). Hence $G_{\mathsf{F}_3} = A_6$.
	\end{example}

	Note that $S_n$ acts on the set of labeled forests having exactly $n$ vertices by relabeling their vertices. We show now that this action conjugates the chain groups.
	
	\begin{proposition} \label{prop:order groups of isomorphic and dual posets}
		If $\ff$ and $\ff'$ are forests with $n$ vertices that are a relabeling version of each other, then their chain groups are conjugate in $S_n$.
	\end{proposition}
	
	\begin{proof}
		Let $G$ and $G'$ denote the chain groups of $\ff$ and $\ff'$, respectively. Let $\pi \colon \ff \to \ff'$ be a graph isomorphism. In particular, we can interpret $\pi$ as an element in $S_n$. Consider the map $\varphi \colon G \to G'$ defined by $\varphi(\sigma) = \pi \sigma \pi^{-1}$. For every maximal path $i_1, \dots, i_m$ in $\ff$, we have that $\pi(i_1), \dots, \pi(i_m)$ is a maximal path in $\ff'$ and, therefore,
		\[
			\varphi((i_1 \ \dots \ i_m)) = \pi (i_1 \ \dots \ i_m) \pi^{-1} = (\pi(i_1) \ \dots \ \pi(i_m))
		\]
		is a maximal path in $\ff'$. So the map $\varphi$ is well defined. It follows immediately that $\varphi$ is a group homomorphism. Now we can define $\psi \colon G' \to G$ by $\psi(\sigma) = \pi^{-1} \sigma \pi$, and similarly verify that it is a well-defined homomorphism of groups. Since $\varphi$ and $\psi$ are inverses of each other, $\pi$ is an isomorphism.
	\end{proof}

	Proposition~\ref{prop:order groups of isomorphic and dual posets} gives us the freedom to talk about the chain group of a non-necessarily labeled forest as long as we are not interested in the specific subgroup of the symmetric group we are dealing with but only in its isomorphic class.
	
	Let us verify now that the chain group of a forest is the direct product of the chain groups of the trees of the given forest.
	
	\begin{proposition} \label{prop:chain group of disjoint graphs}
		If $\ff$ is a forest which is the disjoint union of the trees $\sft_1, \dots, \sft_m$, then $G_{\ff} \cong G_{\sft_1} \times \dots \times G_{\sft_m}$.
	\end{proposition}
	
	\begin{proof}
		Let $n = |\ff|$. It suffices to assume that $m=2$. Let $\sigma_1, \dots, \sigma_r$ be the generating cycles induced by the maximal paths of $\sft_1$, and let $\rho_1, \dots, \rho_s$ be the generating cycles induced by the maximal paths of $\sft_2$. As $\sigma_i$ and $\rho_j$ are disjoint cycles in $S_n$ for each pair $(i,j) \in [r] \times [s]$, we can write every element of $G_{\ff}$ as $\sigma \rho$ for some $\rho \in G_{\sft_1}$ and $\rho \in G_{\sft_2}$. Now it immediately follows that the assignment $(\sigma, \rho) \mapsto \sigma \rho$ is, indeed, an isomorphism from $G_{\sft_1} \times G_{\sft_2}$ to $G_{\ff}$.
	\end{proof}

%
	
%
%

	\section{Abelian Chain Groups associated to $n$-Forests} \label{sec:abelian case}
	
	In this section we characterize the forests whose chain groups are abelian. In addition, we determine those abelian groups that show up as chain groups of some forest.
	
	\begin{example} \label{ex:the cyclic group of order n is always a chain group}
		Let $\sft$ be an $n$-tree with at most two leafs. Then there is only one maximal path, namely $\sigma(1) \prec \dots \prec \sigma(n)$ for some bijection $\sigma \colon [n] \to [n]$. Thus, the chain group associated to $\sft$ is $G = \langle (\sigma(1) \ \dots \ \sigma(n)) \rangle \cong \zz_n$.
	\end{example}
	
	More generally, we have the following result.
	
	\begin{proposition}
		Let $n$ be a natural, and let $\ff$ be an $n$-forest. Then the associated chain group of $\ff$ is abelian if and only if $\ff$ is the disjoint union of paths.
	\end{proposition}
	
	\begin{proof}
		Example~\ref{ex:the cyclic group of order n is always a chain group}, along with Proposition~\ref{prop:chain group of disjoint graphs} in the preview section, immediately implies that if $\ff$ is the disjoint union of $k$ chains of lengths $n_1, \dots, n_k$, then $G_{\ff} \cong \zz_{n_1} \times \dots \times \zz_{n_k}$. In particular, $G_{\ff}$ is abelian. To prove the direct implication, suppose by contradiction that there are two distinct maximal paths $(i_1, \dots, i_r)$ and $(j_1, \dots, j_s)$ that are not disjoint. Set $\sigma = (i_1 \ \dots \ i_r)$ and $\tau = (j_1 \ \dots \ j_s)$. As $\ff$ has no cycles, $\{i_1, i_r\} \neq \{j_1, j_s\}$. We can assume, without loss of generality that $i_r \neq j_s$. Because $i_r \neq j_s$, there exists an index $p > 1$ such that $i_p \in \{j_1, \dots, j_s\}$ and $i_{p+1} \notin \{j_1, \dots, j_s\}$ (let $i_p = j_q$). In this case $(\sigma \circ \tau \circ \sigma^{-1})(i_p) = i_{p+1} \neq j_{q+1} = \tau(j_q)$. Hence $\sigma$ and $\tau$ do not commute, contradicting the fact that $G_{\ff}$ is abelian.
	\end{proof}
	
	For $n \in \nn$, we study which abelian groups are chain groups associated to $n$-forests.
	
	\begin{proposition} \label{prop:elementary abelian groups that are chain groups}
		The elementary abelian group $(\zz/p\zz)^r$, where $p$ is prime and $r$ is a natural, is the chain group associated to an $n$-forest if and only if $rp \le n$.
	\end{proposition}
	
	\begin{proof}
		For the direct implication, suppose that $(\zz/p\zz)^r$ is the chain group of an $n$-forest. This implies that $S_n$ contains a copy $G$ of $(\zz/p\zz)^r$. Consider the set $S$ of $p$-cycles inside any disjoint-cycle decomposition of any element of $G$. Take a maximal subset $\{\sigma_1, \dots, \sigma_s\}$ of $S$ satisfying that no element is a power of another one. As $G$ is abelian the $\sigma_i$'s are pairwise disjoint. In addition, $G' = \langle \sigma_1, \dots, \sigma_s \rangle$ is isomorphic to $ \cong (\zz/p\zz)^s$ and contains $G$, which yields that $r \le s$. As the $s$ $p$-cycles are disjoint, one finds that $rp \le sp \le n$.
		
		Suppose, on the other hand, that $rp \le n$. Consider the forest $\ff$ having $r + n - rp$ connected components, $r$ of them being path graphs on $p$ vertices and $n - rp$ of them being $1$-vertex trees. The chain group $G$ of $\ff$ is generated then by $r$ disjoint $p$-cycles. Hence $G$ is a subgroup of $S_n$ isomorphic to the elementary abelian group $(\zz/p\zz)^r$, which completes the proof.
	\end{proof}
	
	Not every abelian subgroup of $S_n$ can be reached as an associated chain group of an $n$-forest. In particular, the abelian subgroups of maximal order are never achieved in this way, as we shall prove in Theorem~\ref{thm:abelian group of maximum order that are chain groups}. The following theorem describes the abelian subgroups of $S_n$ of maximum order.
	
	\begin{theorem}\cite[Theorem 1]{BG89}
		Let $G$ be an abelian subgroup of maximal order of the symmetric
		group $S_n$. Then
		\begin{enumerate}
			\item $G \cong (\zz/3\zz)^k$ if $n = 3k$;
 			\item $G \cong \zz/2\zz \times (\zz/3\zz)^k$ if $n = 3k+2$;
			\item either $G \cong \zz/4\zz \times (\zz/3\zz)^{k-1}$ or $G \cong (\zz/2\zz)^2 \times (\zz/3\zz)^{k-1}$ if $n = 3k+1$.
		\end{enumerate}
	\end{theorem}

	We can use Theorem~\ref{thm:chain group of the star graph} to argue the following proposition.
	
	\begin{theorem} \label{thm:abelian group of maximum order that are chain groups}
		The maximum order abelian subgroups of $S_n$ are the chain group of $n$-forests.
	\end{theorem}
	
	\begin{proof}
		Suppose first that $n=3k$. Theorem~\ref{thm:chain group of the star graph} guarantees that any maximum order abelian group $G$ of $S_n$ is a copy of $(\zz/3\zz)^k$. It follows by Proposition~\ref{prop:elementary abelian groups that are chain groups} that $G$ is the chain group of an $n$-forest.
		
		Assume now that $n=3k+2$. Consider the $n$-forest $\ff$ consisting of the following $k+1$ connected components: $k$ $3$-vertex paths and one two-vertex path. The chain group associated to $\ff$ is isomorphic to $\zz/2\zz \times (\zz/3\zz)^k$, which is a maximum order abelian group of $S_n$ by Theorem~\ref{thm:chain group of the star graph}.
		
		Lastly, assume that $n = 3k+1$. Then consider the $n$-forest $\ff_1$ having as connected components $k-1$ three-vertex paths and one $4$-vertex path, and also consider the $n$-forest $\ff_2$ having as connected components $k-1$ three-vertex paths and two $2$-vertex path. Notice the chain groups of $\ff_1$ and $\ff_2$ are isomorphic to $\zz/4\zz \times (\zz/3\zz)^{k-1}$ and $(\zz/2\zz)^2 \times (\zz/3\zz)^{k-1}$, respectively. As before such groups have both maximum orders by Theorem~\ref{thm:chain group of the star graph}, and the result follows.
	\end{proof}

	\section{Chain Groups of some Trees} \label{sec:the chain group of a tree}
	
	In this section, we will only consider trees. The simplest family of trees consists of \emph{chains} (i.e., trees containing exactly one maximal path), and chain groups of chains are cyclic. Another very simple example of trees are the ones having all their vertices except one having degree $1$ (see the second graph in Figure~\ref{fig:chain groups of three forests}). It turns out that the chain group of this family of trees is always $A_n$ as the next theorem indicates.
	
	\begin{theorem} \label{thm:chain group of the star graph}
		For every $n \in \nn_{\ge 3}$, there exists a labeled tree with $n$ vertices whose chain group is $A_n$.
	\end{theorem}
	
	\begin{proof}
		When $n = 3$, the alternating group $A_n$ is is isomorphic to $\zz_3$, and it is enough to take $\sft$ to be the only tree on $3$ vertices. Assume that $n \ge 4$. From the fact that every $3$-cycle $(i \ j \ k)$ in $S_n$ not containing $1$ satisfies $(i \ j \ k) = (1 \ i \ j)(1 \ j \ k)$ and the fact that every $3$-cycle $(1 \ i \ j)$ in $S_n$ not containing $2$ satisfies $(1 \ i \ j) = (1 \ 2 \ j)^2(1 \ 2 \ i)(1 \ 2 \ j)$, we can immediately deduce that $A_n$ is generated by $3$-cycles of the form $(1 \ 2 \ i)$ for all $i \in [n] \setminus \{1,2\}$. Now we just need to take $\mathsf{T}$ to be the star graph $K_{1,n-1}$ to have that $G_{\sft} = A_n$ (see, for an illustration, the central forest in Figure~\ref{fig:chain groups of three forests}).
	\end{proof}

	Now we turn to find a large family of forests each of its members $\sft$ has chain group $S_n$, where $n = |V(\sft)|$. First, let us introduce the following definition.
	
	\begin{definition}
		We say that a tree is an \emph{antenna} if it has exactly one vertex of degree three and exactly one maximal path of length two.
	\end{definition}

	It is not hard to verify that if a tree is an antenna, then it must be like $\ff_2$ in Figure~\ref{fig:chain groups of three forests} with, perhaps, the vertical path more prolonged upward. In particular, an antenna has exactly three maximal paths.
	
	\begin{proposition} \label{prop:chain group of an antenna}
		Let $\sft$ be a labeled antenna with an odd number $n$ of vertices. Then the chain group of $\sft$ is $S_n$. 
	\end{proposition}

	\begin{proof}
		By Proposition~\ref{prop:order groups of isomorphic and dual posets}, we can relabel $\sft$ if necessary so that its labels look like the one in Figure~\ref{fig:horizontal antenna}. Let $G_{\sft}$ be the chain group of $\sft$.
		\begin{figure}[h]
			\centering
			\includegraphics[width = 6cm]{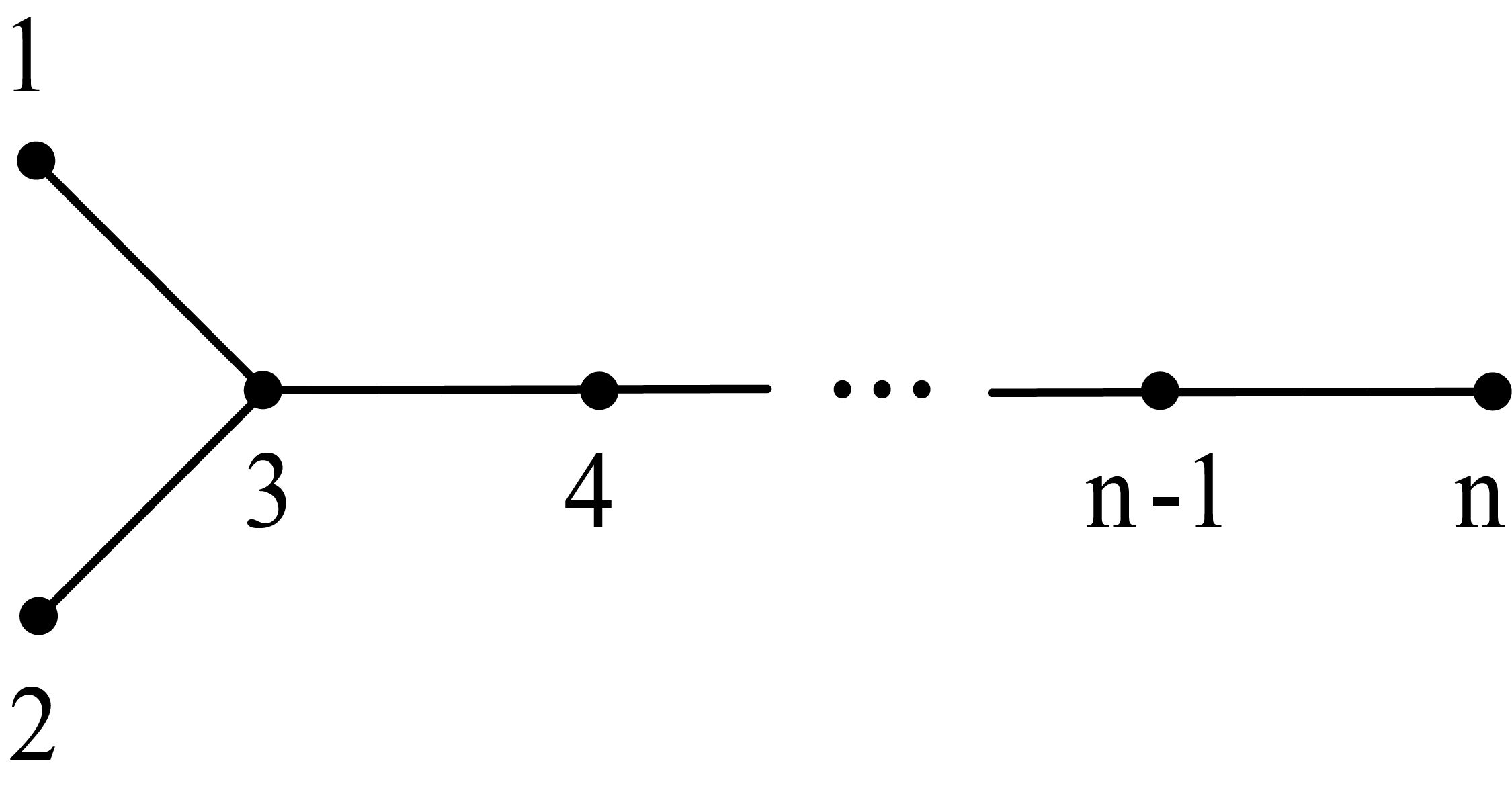}
			\caption{}
			\label{fig:horizontal antenna}
		\end{figure}
		The maximal path of $\sft$ are the $1,3,2$ with corresponding generator $\sigma = (1 \ 3 \ 2)$,  the path $1,3,4, \dots, n$ with corresponding generator $\sigma_1 = (1 \ 3 \ 4 \ \dots \ n)$, and the path $2,3, \dots, n$ with corresponding generator $\sigma_2 = (2 \ 3 \ \dots \ n)$. Notice that $\sigma_1 \circ \sigma_2$ is a cycle of length $n$. In addition, the disjoint cycle decomposition of $\sigma \circ \sigma_2^{-1}$ contains exactly a cycle of length two and a cycle of length $n-2$. Therefore $(\sigma \circ \sigma_2^{-1})^{n-2}$ is a transposition. As $G_{\sft}$ contains a full cycle and a transposition, it must be $S_n$.
	\end{proof}

	Proposition~\ref{prop:chain group of an antenna} says in particular that for every odd $n \ge 5$ there is a tree whose chain forest is $S_n$. In addition, we can use this proposition to find the chain group of more complex forests. Before explaining how to do this, let us introduce the following definition.
	
	\begin{definition}
		Let $G$ be a graph, and let $G'$ be a subgraph of $G$. We say that $G'$ is an \emph{extended subgraph} of $G$ is every leaf of $G'$ is also a leaf of $G$. 
	\end{definition}

	\emph{Extended subtrees} and \emph{extended subforests} are defined in a similar fashion. Notice that a connected component of a graph is always an extended subtree. The next figure depicts a tree and two of its subgraphs (which happen to be forests) only one of them being extended.
	
	\begin{figure}[h]
		\centering
		\includegraphics[width = 2.5cm]{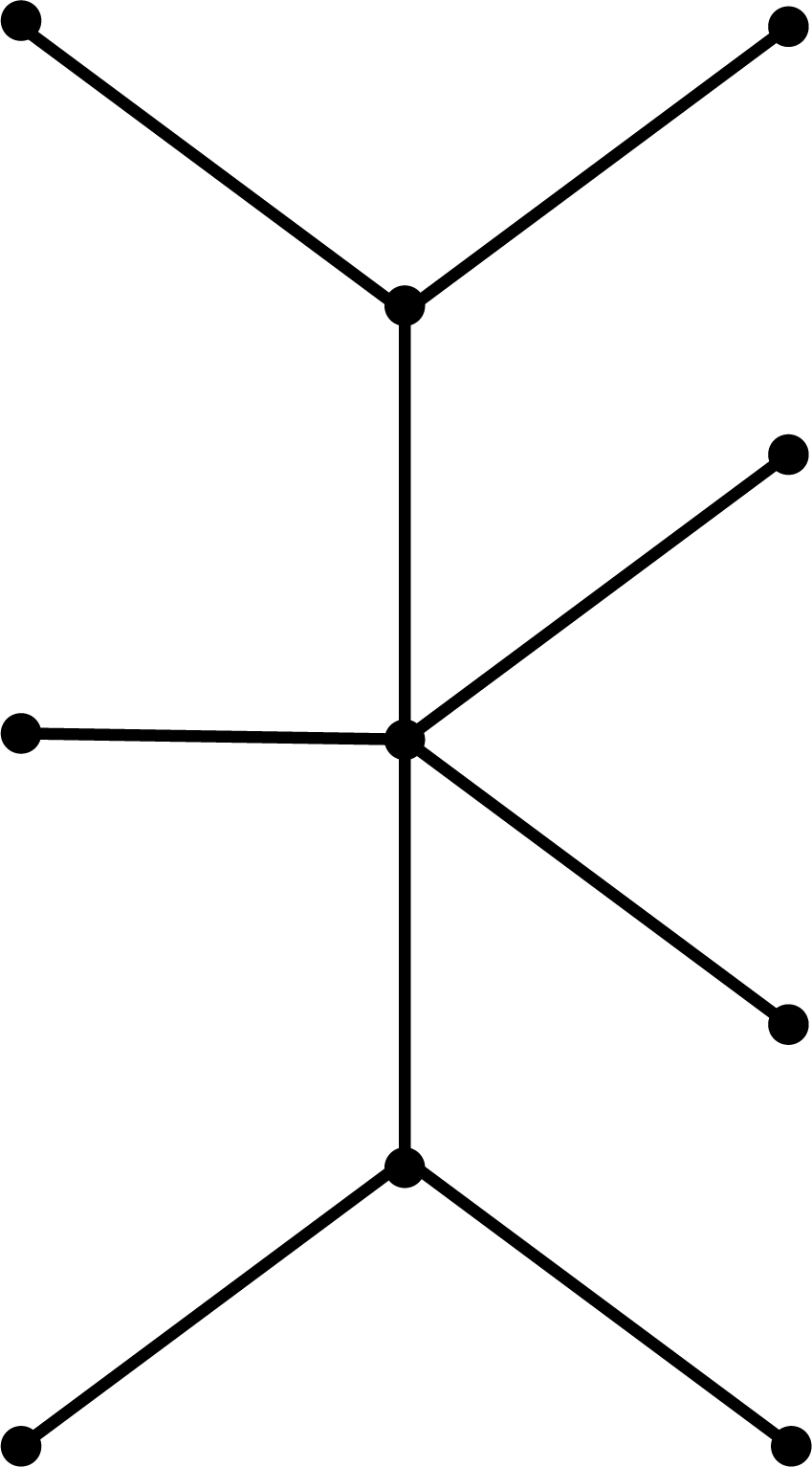} \hspace{2cm}
		\includegraphics[width = 2.5cm]{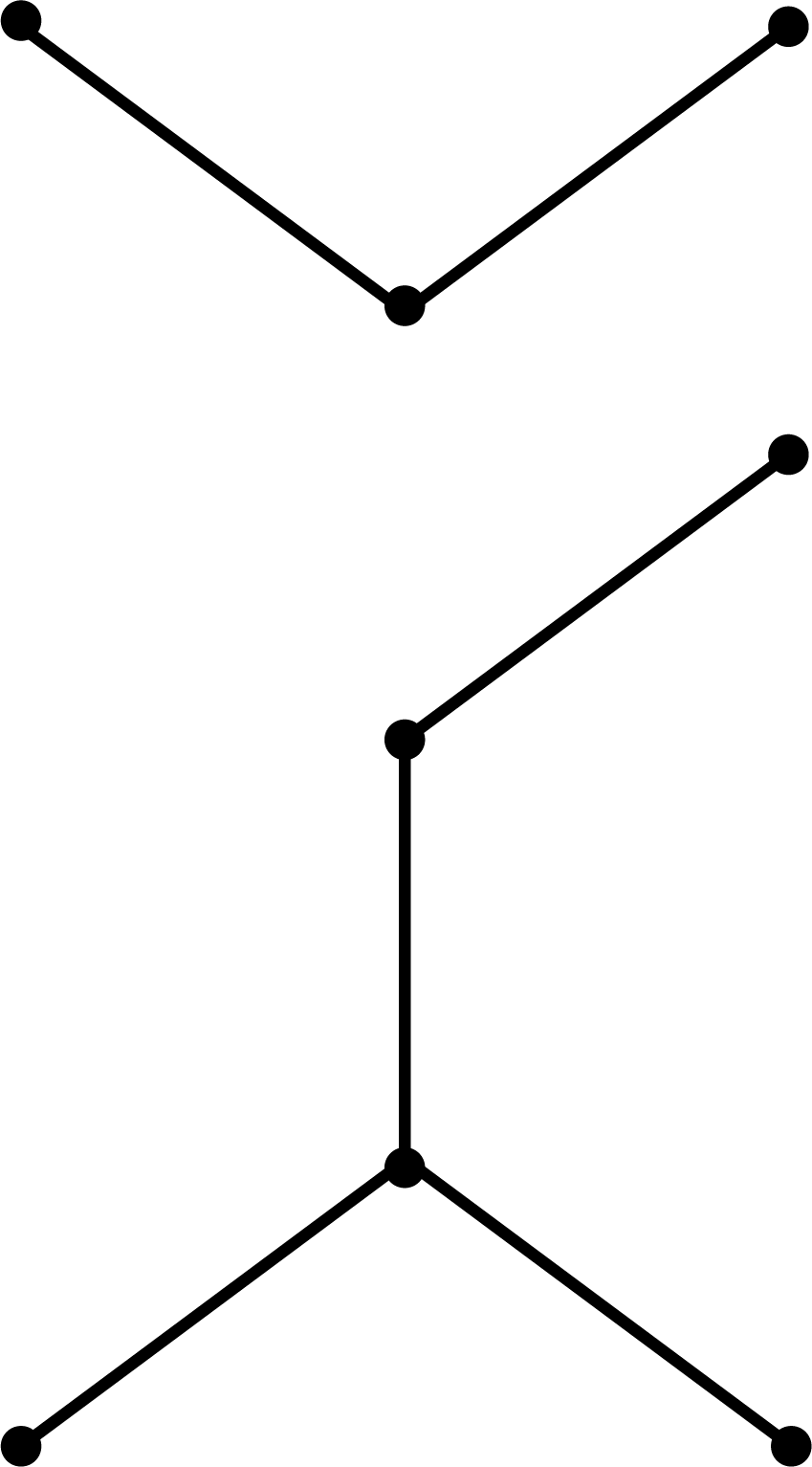} \hspace{2cm}
		\includegraphics[width = 2.5cm]{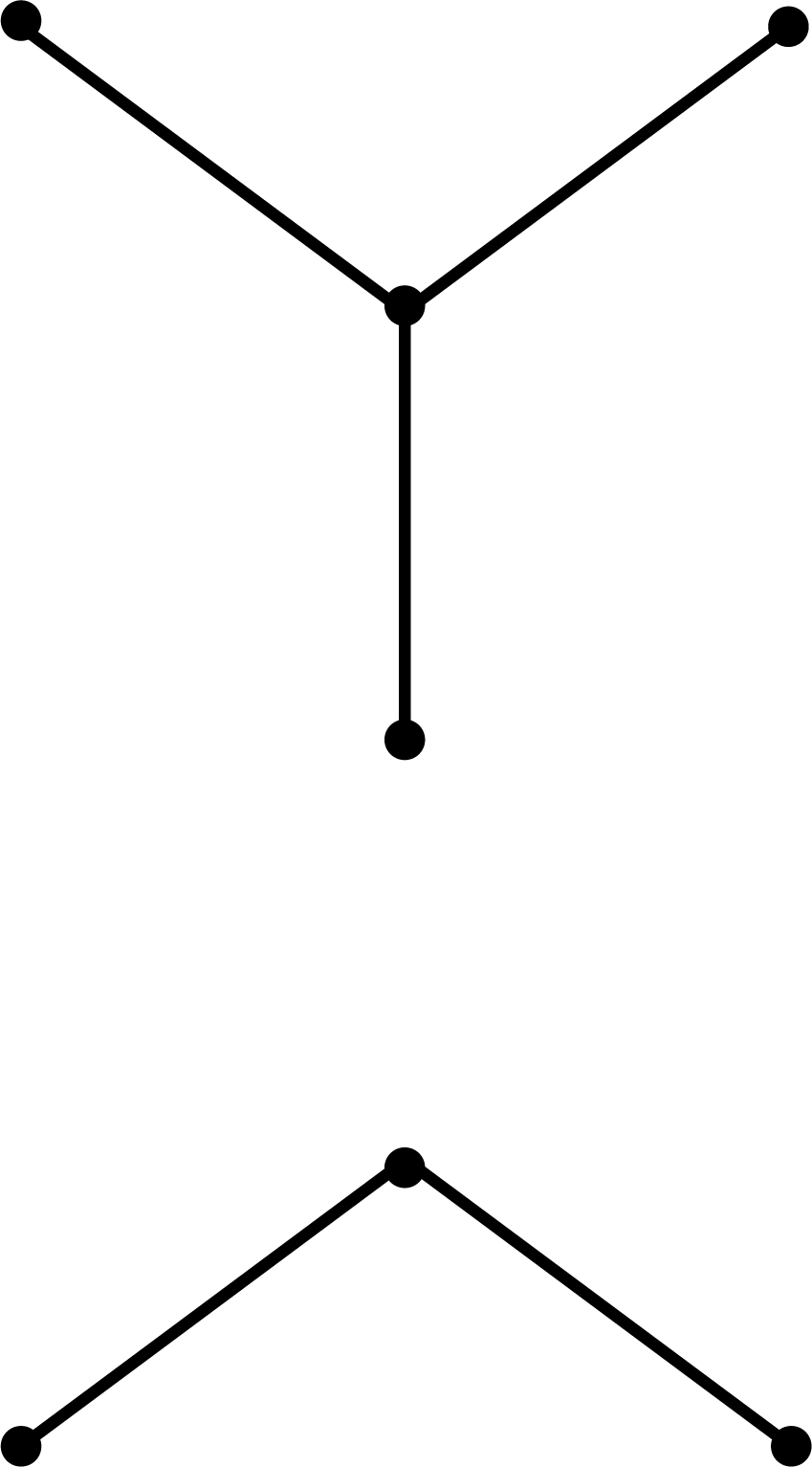}
		\caption{A tree and two of its subforests; only the subforest in the center is extended.}
		\label{fig:extended subforest}
	\end{figure}

	It follows immediately that if $\ff$ is a forest and $\ff'$ is an extended subforest of $\ff$, then the chain group of $\ff'$ is a subgroup of the chain group of $\ff$. Using this observation and Proposition~\ref{prop:chain group of an antenna} is not hard to argue the following result.
	
	\begin{proposition} \label{prop:trees with full chain group}
		Let $\sft$ be a tree with $n$ vertices and a maximal path $\alpha$ of length two. If the distance from any of the two leaves in $\alpha$ to any leaf that is not in $\alpha$ is odd, then the chain group of $\sft$ is $S_n$.
	\end{proposition}

	\begin{proof}
		Left to the reader.
	\end{proof}

	Proposition~\ref{prop:trees with full chain group}, along with Proposition~\ref{prop:chain group of disjoint graphs}, allows us to easily determine the chain groups of relatively complex forests. For example, the chain group of the forest illustrated in Figure~\ref{fig:final forest} is $S_{12} \times S_{12} \times S_{17}$.
	\begin{figure}[h]
		\centering
		\includegraphics[width = 2.5cm]{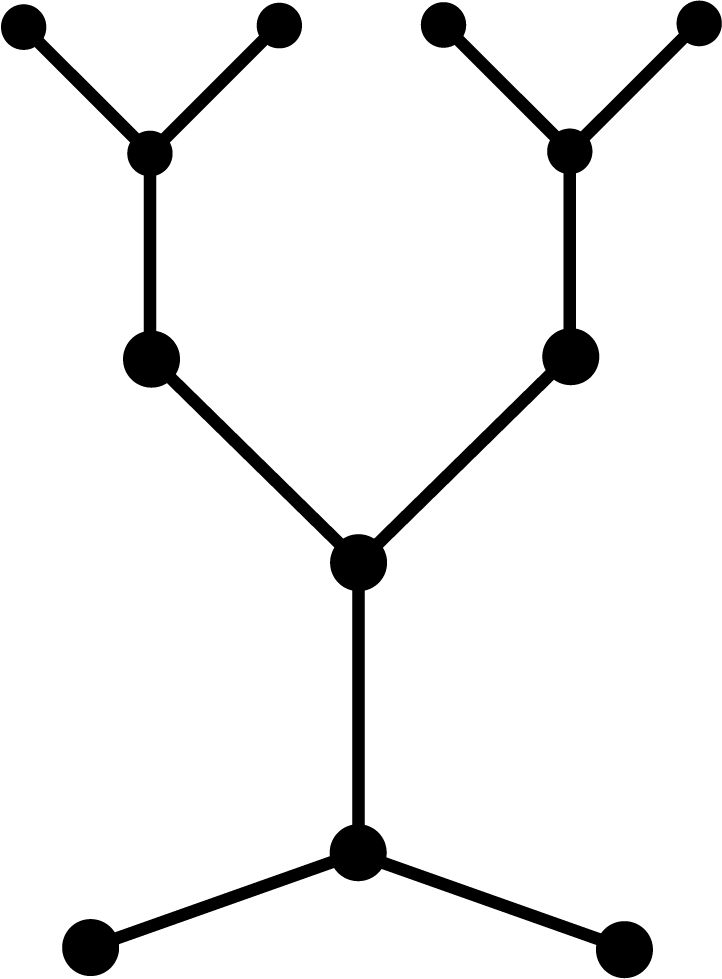} \hspace{1cm}
		\includegraphics[width = 4.5cm]{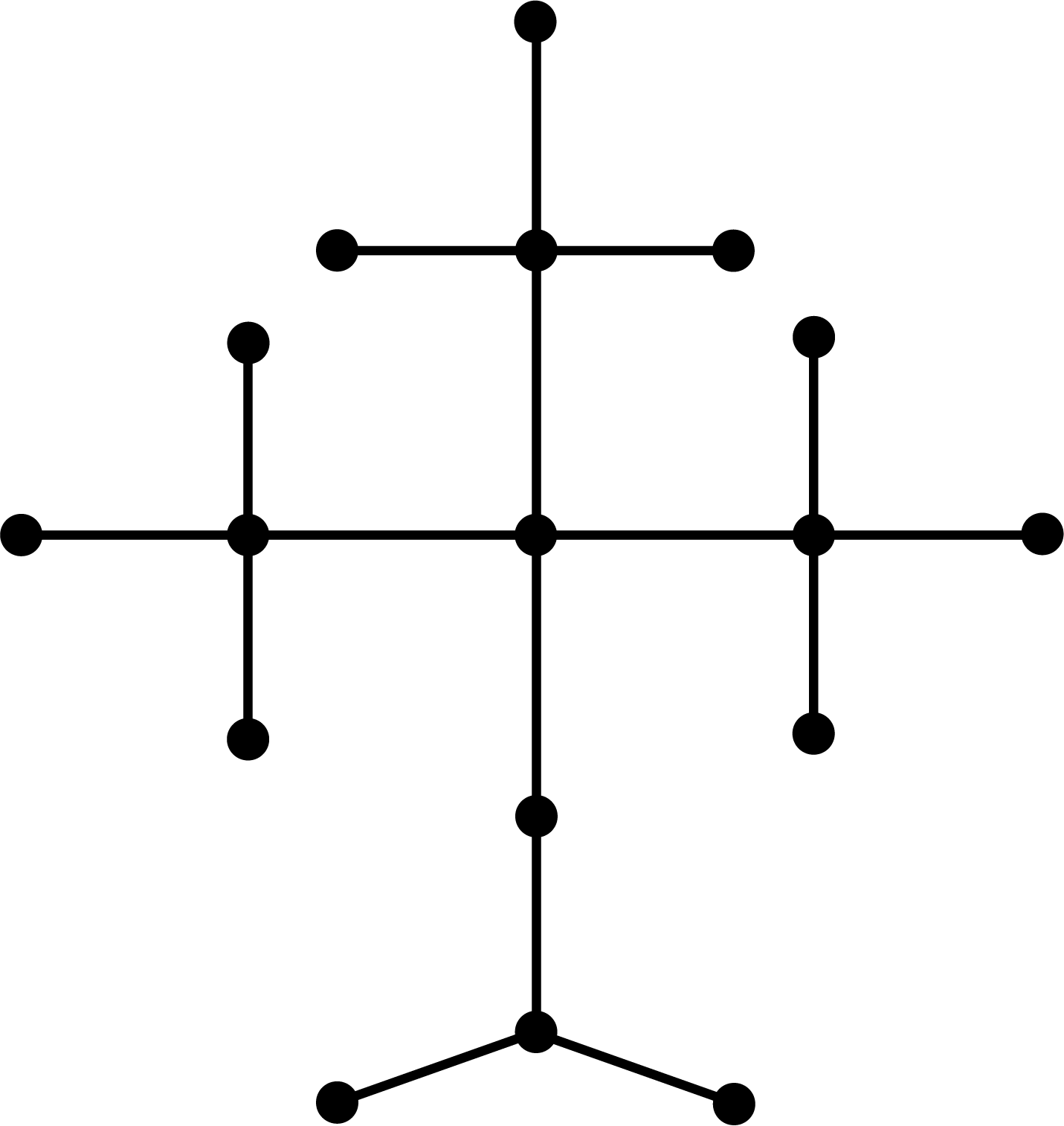} \hspace{1cm}
		\includegraphics[width = 2.5cm]{FinalTree2.png}
		\caption{A forest with 41 vertices.}
		\label{fig:final forest}
	\end{figure}

	We closed this section providing a sufficient condition for the chain groups of some antenna-like trees to have a full cycle.
	
	\begin{proposition}
		Let $\sft$ be a labeled tree with $n$ vertices having exactly one vertex $v$ of degree three and the rest of its vertices of degree at most two. If $\mathsf{d}(v,w)$ is odd for some leave $w$, then $G_{\sft}$ has a cycle of length $n$.
	\end{proposition}

	\begin{proof}
		Note first that $\sft$ only has three leaves, say $w_1, w_2,$ and $w_3$. Suppose, without loss of generality, that $\mathsf{d}(v,w_1)$ is odd. Let $w_1 = v_1, v_2, \dots, v_{2t} = v$ be the path from $w_1$ to $v$. Also, let $v = v_{2t+1}, v_{2t+2}, \dots, v_r = v_2$ and $v = v_{2t+1}, v'_1, \dots, v'_s = w_3$. Then notice that
		\[
			(v_1 \ \dots \ v_r) \circ (v_1 \ \dots \ v_{2t+1} \ v'_1 \ \dots \ v'_s) = (v_1 \ v_3 \ \dots \ v_{2t+1} \ v'_1 \ \dots \ v'_s \ v_2 \ v_4 \ \dots \ v_{2t} \ v_{2t+2} \ \dots v_r)
		\]
		is a cycle of length $n$.
	\end{proof}
	
	\section{The Dihedral is Missing} \label{sec:missing chain groups}
	
	The symmetric group $S_n$ contains many copies of the dihedral group $D_{2n}$. However, none of these copies is the chain group of any labeled forest with $n$ vertices.
	
	\begin{lemma} \label{lem:two vertices greater than two}
		Let $\sft$ be a tree with at least two vertices whose degree is strictly greater than $2$. Then $\sft$ contains a maximal chain $C$ such that $|\mathsf{T} \! \setminus \! C| \ge 3$.
	\end{lemma}
	
	\begin{proof}
		Let $v$ and $w$ be two distinct vertices of $\sft$ with degrees at strictly greater than $2$. Let $\rho$ be the unique path in $\sft$ from $v$ to $w$. As $\deg(v) \ge 3$ there exist two maximal paths $\nu_1$ and $\nu_2$ among those starting at $v$ such that $\nu_1 \cap \nu_2 = \nu_1 \cap \rho = \nu_2 \cap \rho = \{v\}$. Similarly, there are two paths $\omega_1$ and $\omega_2$ maximal among those starting at $w$ satisfying that $\omega_1 \cap \omega_2 = \omega_1 \cap \rho = \omega_2 \cap \rho = \{v\}$. Let $v_1,v_2,w_1,w_2$ be the leaves contained in $\nu_1, \nu_2, \omega_1, \omega_2$, respectively. Now take $C$ to be the unique maximal chain from $v_1$ to $v_2$. Because $C$ does not contain any vertex in $\{w,w_1,w_2\}$, the lemma follows.
	\end{proof}
	
	\begin{theorem}
		For every $n \in \nn$, the dihedral group $D_{2n}$ is not a chain group of any labeled forest with $n$ vertices.
	\end{theorem}
	
	\begin{proof}
		The dihedral $D_2 \cong \zz_2$ cannot be the chain group of the trivial forest because the latter is trivial. In addition, the possible chain groups of a $2$-forest are isomorphic to either the trivial group or $\zz_2$ and $D_4 \cong V_4$; therefore the theorem is also true in the case of $n=2$. Let $n \ge 3$ and assume, by way of contradiction, that $\ff$ is an $n$-forest whose chain group is the dihedral $D_{2n}$.
		
		First, let us consider the case in which $\ff$ is disconnected. Since the action of $D_{2n}$ on $[n]$ does not fix any point, $\ff$ cannot have trivial connected components (i.e., isolated vertices). If $\ff$ had a connected component $C$ with at least three vertices, then any element of $D_{2n}$ associated to a maximal path of a component $C' \neq C$ would fix at least three elements of $[n]$, namely the vertices of $C$, which is impossible because every nontrivial element of $D_{2n}$ fixes at most two elements of $[n]$. Therefore every connected component of $\ff$ contains exactly two vertices. But the fact that $\ff$ is the disjoint union of paths, contradicts that $D_{2n}$ is not abelian. Hence $\ff$ cannot be disconnected.
		
		Now let $\ff$ be a tree with $n$ vertices whose associated chain group is $D_{2n}$. Since $D_{2n}$ is not abelian, $\ff$ is not a path graph. If $\ff$ contains two vertices of degree strictly greater than $2$, then Lemma~\ref{lem:two vertices greater than two} guarantees the existence of a maximal chain $C$ such that $\ff \! \setminus \! C$ contains at least three vertices. Thus, the generator of $D_{2n}$ associated to the chain $C$ would fix at least three elements of $[n]$, which cannot be possible. Hence $\ff$ must contain at most one vertex $v$ such that $\deg(v) > 2$.
		
		Suppose first that $\deg(v) \ge 4$. If $\deg(v) \ge 5$, then it is not hard to see that for every maximal chain $C$ of $\ff$ containing $v$ one has $|\ff \setminus C| \ge 3$, which would imply that the generator of $D_{2n}$ associated to $C$ fixes at least $3$ elements of $[n]$. Thus, assume $\deg(v) = 4$. If $|\ff| > 5$, then it follows as before that there are at least $3$ vertices in the complement of any maximal chain of $\ff$ having minimum size among those containing $v$. On the other hand, $|\ff| = 5$ implies that $\ff$ is isomorphic to $K_{1,4}$. By Theorem~\ref{thm:chain group of the star graph}, the chain group of $K_{1,4}$ is $A_5$, which is not isomorphic to $D_{10}$ (for instance, $|A_5| > |D_{10}|$), a contradiction.
		
		Finally, suppose that $\deg(v) = 3$. To argue this case, let $C$ be a maximal chain of $\ff$ with maximum cardinality among those containing $v$, and let $\rho$ be the generator of the copy of $D_{2n}$ in $S_n$ induced by $C$. The element $\rho$ is not an $n$-cycle as $|C| < n$. On the other hand, the maximality of $|C|$ implies that $\rho$ has order strictly greater than $n/2$. As the only elements in $D_{2n}$ of order $n$ are the $n$-cycles, we obtain a contradiction. The theorem now follows.
	\end{proof}

\end{document}